\documentclass[12pt]{article}
\usepackage[ansinew]{inputenc}

\usepackage{color}
\usepackage{times}
\usepackage[hidelinks]{hyperref}
\usepackage{enumerate,latexsym}
\usepackage{latexsym}
\usepackage{amsmath,amssymb}
\usepackage{graphicx}

\usepackage{amsmath,amsthm,amsfonts,amssymb}
\usepackage{graphicx}
\usepackage[margin=3cm]{geometry}

\makeatletter
\def\namedlabel#1#2{\begingroup
 #2%
 \def\@currentlabel{#2}%
 \phantomsection\label{#1}\endgroup
}
\makeatother

\theoremstyle{plain}
\newtheorem*{theorem*}{Theorem}
\newtheorem*{thmex*}{Theorem~\ref{example}}
\newtheorem*{thmasymp*}{Theorem~\ref{thmAsymp}}
\newtheorem{theorem}{Theorem}[section]

\newtheorem{corollary}[theorem]{Corollary}

\newtheorem{lemma}[theorem]{Lemma}

\newtheorem{proposition}[theorem]{Proposition}

\theoremstyle{definition}
\newtheorem{definition}[theorem]{Definition}

\newtheorem{remark}[theorem]{Remark}
\newtheorem*{acknowledgements*}{Acknowledgements}

\DeclareMathOperator{\arcsinh}{arcsinh}
\DeclareMathOperator{\arccosh}{arccosh}

\newcommand{\de}{\delta}
\newcommand{\es}{\emptyset}
\newcommand{\R}{\mathbb{R}}
\newcommand{\C}{\mathbb{C}}
\newcommand{\D}{\Delta}
\renewcommand{\l}{\lambda}

\newcommand{\N}{{\mathbb{N}}}
\newcommand{\sn}[1]{{\mathbb{S}^{#1}}}
\newcommand{\hn}[1]{{\mathbb{H}^{#1}}}
\newcommand{\hr}{\mathbb H^2\times\mathbb R}
\newcommand{\hh}{\mathbb H^2}

\renewcommand{\O}{\Omega}
\newcommand{\G}{\Gamma}

\newcommand{\ben}{\begin{enumerate}}
\newcommand{\een}{\end{enumerate}}
\newcommand{\wt}{\widetilde}
\newcommand{\g}{\gamma}

\newcommand{\cR}{{\mathcal R}}

\newcommand{\hnr}{{\hn{2}\times \R}}

\newcommand{\ve}{\varepsilon}

\newcommand{\bbD}{\mathbb{D}}
\newcommand{\abs}[1]{\vert #1 \vert}

\newcommand{\ed}{\end{document}}
\renewcommand{\S}{\Sigma}

\newcommand{\ol}{\overline}
\newcommand{\wh}{\widehat}
\definecolor{rrr}{rgb}{.9,0,.1}

\definecolor{rr}{rgb}{.8,0,.3}
\newcommand{\ekt}{\mathbb{E}(-1,\tau)}

\newcommand{\pai}{\partial_\infty}

\renewcommand{\L}{\Lambda}
\usepackage{pdfsync}

\begin{document}
\title{On the asymptotic Plateau problem for area minimizing surfaces in $\ekt$.}
\author{P. Klaser \and A. Menezes
 \and A. Ramos\thanks{First and third authors were partially 
supported by CNPq/Brazil, grant number 406431/2016-7}}
\date{}
\maketitle

\begin{abstract}
We prove some existence and non-existence results for
complete area minimizing surfaces in 
the homogeneous space $\ekt$. As one of our main results,
we present sufficient conditions for a curve $\Gamma$
in $\partial_{\infty} \mathbb{E}(-1,\tau)$ to
admit a solution to the asymptotic Plateau problem, 
in the sense that there exists a complete area minimizing
surface in $\mathbb{E}(-1,\tau)$ having $\Gamma$ as its asymptotic boundary.

\vspace{.15cm}

\vspace{.1cm}


\end{abstract}

\section{Introduction.}

In the last few years the asymptotic Plateau problem in the homogeneous space $\hr$ has been actively studied. For instance, Nelli and Rosenberg \cite{NeRo} proved that for any given Jordan curve $\Gamma\subset\partial_\infty\hr$ that is a graph over $\partial_\infty\hh$ there exists an entire minimal graph $\Sigma$ with $\Gamma$ as its asymptotic boundary; in particular, $\Sigma$ is area minimizing. Sa Earp and Toubiana \cite{setou} also considered the asymptotic Plateau problem in $\hr$ and they showed a general non existence result (see Theorem 2.1 in \cite{setou}) and got as a consequence that there is no complete properly immersed minimal surface whose asymptotic boundary is a Jordan curve homologous to zero in $\partial_\infty\hr$ contained in an open slab between two horizontal circles of $\partial_\infty\hr$ with height equal to $\pi.$ 

Kloeckner and Mazzeo \cite{KloMa} worked with a more general class of curves in the asymptotic boundary of $\hr$ (considering different compactifications of the space) and got a good characterization of curves $\Gamma$ for which there exists a minimal surface that has $\Gamma$ as its asymptotic boundary (see, for instance, Proposition 4.4 and Theorem 4.5 in \cite{KloMa}).

For the Plateau problem involving two closed curves (not homotopically trivial) in the asymptotic boundary of $\hr$, Ferrer, Mart\'in, Mazzeo and Rodr\'iguez \cite{FMMR} proved some existence and non existence results for minimal annuli having these two curves as the asymptotic boundary (see Theorem 1.2 and Theorem 5.1 in \cite{FMMR}).

In addition to the aforementioned results,
Coskunuzer~\cite{cosk1} showed that for any {\em tall curve} 
(i.e., a curve with height greater than $\pi$,
see Definition~\ref{defheight} below) in
$\pai \hnr \equiv \sn1\times \R$, there exists 
an area minimizing surface with that curve as the asymptotic boundary. 
He also showed a non existence result for certain curves that 
are not tall. Here, we obtain similar results
to the ones in~\cite{cosk1} in the ambient space $\ekt$, which is the total
space of a fibration over $\hn2$ with bundle curvature
$\tau$.
In particular, when
$\tau =0$, $\mathbb{E}(-1,0)$ is isometric to the Riemannian 
product $\hnr$, which allows us to reobtain and extend some 
of the results of~\cite{cosk1}.

Throughout this work,
unless specified otherwise, we use the cylinder model for $\ekt$.
Specifically, let $\bbD$ denote the unitary open disk
in the complex plane and let, for $\tau \in \R$,
$\ekt = (\bbD\times \R, ds^2_\tau)$, where $ds^2_\tau$ is the metric 
defined by
\begin{equation}\label{metriccyl}
ds^2_\tau = \l^2(dx^2+dy^2)+\left(2\tau\l(ydx-xdy)+dt\right)^2,
\end{equation}
for $\l = \frac{2}{1-x^2-y^2}$.
We consider the asymptotic boundary of $\ekt$ as
being induced by the product topology of $\ol{\bbD}\times \R$,
$\partial_\infty\ekt= (\partial \bbD) \times \R = \sn1\times\R$.
Moreover, if $\S$ is a complete surface immersed in $\ekt$ 
we define the {\em asymptotic boundary of $\S$}
as the set
$$\pai \S = \{(p,t) \in \sn1\times\R 
\mid \exists (p_n,t_n)_{n\in\N} \subset \S
\text{ s.t. } (p_n,t_n) \to (p,t)\}.$$

In order to state our main results,
we next give the definition of {\em height}
of a curve in $\pai\ekt$. We notice that throughout the paper, curves
will be assumed to be piecewise smooth and non-degenerate.

\begin{definition}[Height of a curve]
\label{defheight}
Let $\G$ be a finite collection of pairwise disjoint
simple closed curves in $\partial_\infty\ekt$
and $\O= \pai\ekt\setminus \G$. For each $p\in\sn1$, 
let
$\ell_p = \{p\}\times\R$ denote the vertical line over
$p$ in $\pai \ekt$ and let
$\ell_{p}^1,\,\ell_p^2,\,\ldots,\,\ell_p^{n_p}$
be the connected components of $\O \cap \ell_p$.
For $i\in \{1,\,2,\,\ldots,\,n_p\}$, let $\abs{\ell_p^i}$
denote the (possibly infinite) euclidean length of $\ell_p^i$.
Then, the {\em height of $\G$ at $p$} is
$h_{\Gamma}(p) = \min_{i\in \{1,\ldots,n_p\}}\abs{\ell_p^i}$ and
the {\em height of $\G$} is
$$h(\G) = 
\inf_{p\in\sn1} h_{\Gamma}(p).$$
\end{definition}

\begin{remark}\label{rmheightfun}
As in $\hnr$, an isometry of $\hn2$ induces an isometry in $\ekt$.
Nevertheless,
for $\tau >0$, the induced isometry changes the $t$-coordinate,
as observed in Proposition~\ref{prop-iso}. Since this change
is constant along any fiber,
the vertical distance between two
points in the same fiber is invariant under isometries. 
In particular, the definition of the height of a curve 
is well posed. 
Furthermore, differently from $\hh\times\R$, there is
no intrinsic notion of a 
height function in $\ekt$, $\tau>0$.
An example that shows this dependence is the horizontal 
slice $\{t =0\}\subset \ekt$ in the half-plane model,
which becomes (see~\eqref{eq:mudvar}) a piece of a helicoid in the disk model: 
its height should be constant and
equal to zero in the half plane model 
but it is not constant in the disk model.
To avoid this ambiguity, throughout the paper the height 
of a curve in $\pai\ekt$
is here defined for the cylinder model of $\ekt$.
\end{remark}

We next make precise the notion of a {\em tall curve} 
in $\ekt$. Note that our definition differs slightly from the
one introduced by~\cite[Definition~2.4]{cosk1}, which allows
us to treat a broader class of curves.

\begin{definition}\label{deftall}
Let $\G$ be a finite, pairwise disjoint collection of 
simple closed curves in $\partial_\infty\ekt$. 
We say that 
$\G$ is a {\em tall curve}
if $h_\G(p)>\sqrt{1+4\tau^2}\pi$ for all $p\in \G$. Otherwise, we say that
$\G$ is a {\em short curve}.
\end{definition}

Our first main result is the following.

\begin{theorem}\label{thmmain}
Let $\G \subset \pai \ekt$ be a finite collection of
pairwise disjoint simple
closed curves. If $\G$ is tall, 
there exists a complete, possibly disconnected,
area minimizing
surface $\S$ in $\ekt$ with $\partial_\infty \S = \G$.
\end{theorem}

Note that if $\G\subset \pai \ekt$ is a short curve,
then there exists a point 
$p\in \G$ such that $h_\G(p)\leq \sqrt{1+4\tau^2}\pi$. 
Concerning such curves, 
we expect that, at least for the case where there 
is an open arc $I\subset \sn1$
such that $h_\G(p)\leq\sqrt{1+4\tau^2}\pi$ for all
$p \in I$, there is no area minimizing surface with
asymptotic boundary $\Gamma$. However, this question is still open,
even in the case of $\hnr$. 
In the following result we are able to prove a special situation
of this nonexistence result.

\begin{theorem}\label{thmnonexist}
Let $\G$ be a short curve for which there exists an open 
arc $I\subset \sn1$ where
\begin{equation}\label{eq:boundheight}
h_\G(p)< (\sqrt{1+4\tau^2}-4\abs{\tau})\pi,
\ \ \mbox{for all} \ \ p \in I.
\end{equation}
Then, there is no area minimizing
surface $\S$ in $\ekt$ with $\pai \S = \G$.
\end{theorem}

\begin{remark}
In the case $\tau = 0$, Theorem~\ref{thmnonexist} is equivalent to
the nonexistence result of Coskunuzer~\cite{cosk1}.
When $\tau \neq 0$, it is not clear whether the bound assumed 
in~\eqref{eq:boundheight} is sharp,
and it only gives information for $\abs{\tau}<\frac{1}{\sqrt{12}}$.
This bound is necessary
to our proof since isometries in $\ekt$
do not preserve the $t$-coordinate; see Proposition~\ref{prop-iso},
Corollary~\ref{cor-iso} and Remark~\ref{rm-iso}. 
\end{remark}

Let us mention that after the completion of this paper, 
J. Castro-Infantes~\cite{CI} considered the same asymptotic 
Plateau problem and among other results, using the 
halfplane model for $\hh$, was able to improve the 
constant in Theorem \ref{thmnonexist}.

The organization of the paper is the following.
In Section~\ref{secPrel}, we present some background 
material for the study of minimal surfaces in $\ekt$. 
In Section~\ref{secproofs} we prove our main theorems;
and in Section~\ref{secproof} we prove a technical fact 
used in the proof of Theorem~\ref{thmnonexist}.

\begin{acknowledgements*} The authors would like to thank Baris 
Coskunuzer, Francisco Mart\'in and Magdalena Rodr\'iguez
for useful discussions concerning the topics of this manuscript.
\end{acknowledgements*}

\section{Preliminaries.}\label{secPrel}

Let $\wt{{\rm SL}}(2,\R)$ denote the universal covering of
the special linear group of $2\times 2$ 
real matrices. 
For each $\tau \in \R$ there exists a left invariant metric
$ds^2_\tau$ in $\wt{\rm SL}(2,\R)$
such that $(\wt{\rm SL}(2,\R),ds^2_\tau) = \ekt$ becomes the total space
of a Riemannian fibration over the hyperbolic plane $\hn2$ with
bundle curvature $\tau$. Note that for any $\tau \in \R$ the group of isometries
of $\ekt$ has dimension four (for a nice
discussion about the
$\mathbb{E}(\kappa,\tau)$ spaces, see Daniel~\cite{dan1}). 
A special case to be considered is when $\tau = 0$, where
$\mathbb{E}(-1,0)$ is isometric to the Riemannian product
$\hnr$. 
In particular, all of our results also hold in $\hnr$.
We also note that for $\tau \neq 0$,
the spaces $\ekt$ and $\mathbb{E}(-1,-\tau)$ are isometric,
hence it is without loss of generality that we assume that 
$\tau \geq 0$.

As stated in the Introduction, we use the cylinder model
$(\bbD\times\R,ds^2_\tau)$ to $\ekt$, where $ds^2_\tau$
is given in~\eqref{metriccyl}.
We also let $\pi_1\colon\ekt\to\bbD$ and $\pi_2\colon \ekt \to \R$
be the projections onto the first and the second coordinates,
respectively.

The isometry group of $\ekt$ is generated by the 
lifts of the isometries of the disk model of
$\hh$, together with vertical translations along the fibers (see, for
instance, Theorem~2.9 in \cite{Younes}). Precisely, 
the following holds.

\begin{proposition}
\label{prop-iso}
The isometries of $\ekt$ are given by
\begin{equation}\label{eq:liftf}
F(z,t)=(f(z),t- 2\tau \mbox{arg}f'(z)+c)
\end{equation}
or
\begin{equation}\label{eq:liftg}
G(z,t)=(\overline{f(z)},-t+ 2\tau \mbox{arg}f'(z)+c),
\end{equation} where $f$ is a
positive isometry of the disk model of $\hh$, $c\in \R$ and arg 
$f':\hh\to\R$ is a smooth angle function for $f'$.
\end{proposition}

One of the main difficulties that arises when working in the cylinder
model of $\ekt$ when $\tau \neq 0$ is that isometries do not preserve the
$t$-coordinate. The next result 
gives an upper bound to this gap on the $t$-coordinate for
some isometries of $\ekt$; we make use of this bound
in the proof of our non-existence result.

\begin{corollary}\label{cor-iso}
For any positive isometry $f$ of the disk model of $\hn2$,
there exists an isometry $F\colon\ekt\to \ekt$ such that the projections
$\pi_1$ and $\pi_2$ satisfy, for all $z\in \bbD$ and $t\in\R$,
that $\pi_1(F(z,t)) = f(z)$ and
$\abs{\pi_2(F(z,t)) - t}<2\tau\pi$.
\end{corollary}
\begin{proof}
First, note that any positive isometry of the disk model of $\hn2$ can be
represented by a Möbius transformation
$$f(z) = \frac{w_1z-\ol{w_2}}{w_2z-\ol{w_1}},$$
where $w_1,\,w_2\in \C$ are such that
$\abs{w_1}^2-\abs{w_2}^2 = 1$. In particular, it holds that
\begin{equation}\label{eq:fprime}
f'(z) = \frac{-1}{(w_2z-\ol{w_1})^2} = 
-\ol{w_2}^2\frac{(\ol{z}-o)^2}{\abs{w_2z-\ol{w_1}}^4},
\end{equation}
where $o = \frac{w_1}{\ol{w_2}}$. Note that
$\abs{o} = \frac{1}{\abs{f(0)}} >1$, hence $f'(z) \neq 0$
for all $z\in\bbD$.
For any $\theta_1<\theta_2$, let
$$\L_{\theta_1,\,\theta_2} = \{re^{i\theta}\in \C\mid r>0 \text{ and } 
\theta\in (\theta_1,\,\theta_2)\}.$$
We next analyze the image set $f'(\bbD)$ to show that there
exist $\theta_1<\theta_2$ with
$\theta_2-\theta_1<2\pi$ such that $f'(\bbD) \subset \L_{\theta_1,\theta_2}$.

Let $\wt{\theta} \in[0,2\pi)$ be such that 
$-\ol{w_2}^2 = \abs{w_2}^2e^{i\wt{\theta}}$. Since multiplication by a positive
constant does not change the argument of a complex number,
it follows from~\eqref{eq:fprime} that 
$$f'(z)\in \L_{\theta_1,\theta_2} \quad \iff \quad
(\ol{z}-o)^2\in \L_{\theta_1-\wt{\theta},\theta_2-\wt{\theta}}
 \quad \iff \quad
\ol{z}-o\in \L_{\frac{\theta_1-\wt{\theta}}{2},\frac{\theta_2-\wt{\theta}}{2}}.$$

Note that $\{\ol{z}-o\mid z\in \bbD\}$ is an
open disk in $\C$ with a positive distance $\abs{o}-1$ to the origin.
Hence, there are 
$\varphi_1<\varphi_2$ with $\varphi_2-\varphi_1 < \pi$ such that
$\{\ol{z}-o\mid z\in \bbD\}\subset \L_{\varphi_1,\,\varphi_2}$.
After choosing, for $i = 1,2$,
$\theta_i = 2\varphi_i+\wt{\theta}$, it follows that
$f'(\bbD) \subset \L_{\theta_1,\theta_2}$ with $\theta_2-2\pi<\theta_1<\theta_2$.

This implies that we may choose a branch of the argument function
such that for all $z\in \bbD$, $\arg(f'(z))\in (\theta_1,\theta_2)$.
After letting $c = 2\tau(\theta_1-\pi)$ in~\eqref{eq:liftf},
the result follows.
\end{proof}

\begin{remark}\label{rm-iso}
The bound $2\tau \pi$ on Corollary~\ref{cor-iso} cannot, in general, be improved.
Indeed, $\sup_{z\in\bbD}\abs{\pi_2(F(z,t)) - t}$ depends uniquely
on $\abs{f(0)}$, as shown in the proof of Corollary~\ref{cor-iso}. Moreover,
if $\{f_n\}_{n\in\N}$ is a sequence of isometries such that
$\lim_{n\to \infty} \abs{f_n(0)}= 1$,
then the respective isometries $F_n$ satisfy
$\lim_{n\to \infty}\sup_{z\in\bbD}\abs{\pi_2(F_n(z,t)) - t}
= 2\tau\pi$.
\end{remark}

In the cylinder model to $\ekt$,
both horizontal planes $\{t=t_0\}$ 
and vertical planes (i.e. the inverse image of a 
geodesic of $\hn2$ by $\pi_1$) are minimal (in fact, they are area
minimizing) surfaces.
We next describe some other
families of minimal surfaces in $\ekt$ that will 
be used as barriers throughout this paper.

\subsection{Rotational Catenoids} 
\label{sec-cat}

We first describe a one-parameter family of complete 
(without boundary) minimal annuli in $\ekt$, which plays
a key role in the proof of Theorem~\ref{thmnonexist}.
Such a family was first obtained by
B. Nelli and H. Rosenberg~\cite{NeRo} for the case $\tau =0$ and
extended to the case where $\tau\neq0$
by C. Pe\~nafiel~\cite{p1}.
Each surface in this family is called a {\em catenoid}
of $\ekt$ and is invariant
under the group of isometries corresponding to rotations about
the $t$-axis of the cylinder model.

Following the notation of~\cite{p1}, for any $d>0$ let
$u_d\colon (\arcsinh(d),\infty)\to (0,\infty)$
be defined by
\begin{equation}\label{defud}
u_d(s) = \int_{{\rm arcsinh}(d)}^s
d\sqrt{\frac{1+4\tau^2\tanh^2(\frac{r}{2})}{\sinh^2(r)-d^2}}dr.
\end{equation}
Then, $u_d$ extends continuously to $s = \arcsinh(d)$ by setting
$u_d(\arcsinh(d)) = 0$ and is strictly increasing. Moreover, 
there exists an increasing function 
$d>0\mapsto {\bf h}(d)\in(0,\frac{\pi}{2}\sqrt{1+4\tau^2})$ such that,
for each $d>0$, $\lim_{s\to \infty}u_d(s) = {\bf h}(d)$. It also holds that
$$\lim_{d\to 0^+} {\bf h}(d) = 0,\quad
\lim_{d\to \infty}{\bf h}(d) = \frac{\pi}{2}\sqrt{1+4\tau^2}.$$
Using this notation, for each $d>0$
the catenoid $M_d$ given by~\cite[Propositions 3.6 and 3.9]{p1}
is $M_d = M_d^+\cup M_d^-$, where
$M_d^+$ and $M_d^-$ are the rotational surfaces parameterized by
$$M_d^\pm =\{ (\tanh\left(r/2\right)\cos(\theta),
\tanh\left(r/2\right)\sin(\theta),\pm u_d(r)) 
\mid r\in [\arcsinh(d),\infty),\,\theta \in[0,2\pi)\}.$$

It follows directly from its definition that the
asymptotic boundary of $M_d$ is the union of the two horizontal circles 
$\sn1\times\{-{\bf h}(d)\}$ and $\sn1\times\{{\bf h}(d)\}$.

\subsection{Tall Rectangles}

Here we will present some key properties 
of complete minimal planes in $\ekt$
that are invariant under a one-parameter group of
hyperbolic isometries. These surfaces are
the so-called
{\em tall rectangles}
and were first described in the $\tau =0$
case by Sa Earp and Toubiana~\cite{setou} and extended, when
$\tau\neq0$,
to the halfspace model for $\ekt$ by Folha and
Pe\~nafiel~\cite{fp1}. 
In what follows,
we describe this family in the cylinder model and
prove that they are in fact area minimizing surfaces.

For a fixed $\tau\geq0$, let
$h>\pi\sqrt{1+4\tau^2}$ and $r\in(0,\pi)$ be given and let
$$\g_0 = \left\{\left(-\cos(\theta),-\sin(\theta),
4\tau\arctan\left(\frac{\sin(\theta)}{1+\cos(\theta)}\right)\right)
\mid \theta\in [-r,r]\right\},$$
$$\g_1 = \left\{\left(-\cos(\theta),-\sin(\theta),
h+4\tau\arctan\left(\frac{\sin(\theta)}{1+\cos(\theta)}\right)\right)
\mid \theta\in [-r,r]\right\}.$$
Using this notation, we next prove the following.

\begin{proposition}\label{proptallrec}
There exists an area minimizing plane
$\cR_h(r)\subset \ekt$, invariant under a one-parameter
group of hyperbolic isometries of $\ekt$ and
with asymptotic boundary given by the union of $\g_0,\g_1$ 
and the two vertical segments joining their
endpoints (see Figure~\ref{figrec}).
\end{proposition}
\begin{proof}
When $\tau =0$, the result follows immediately from
Proposition~2.1 of~\cite{setou}, hence we next assume
that $\tau >0$.
We follow the notation of Folha and Pe\~nafiel~\cite{fp1}, where
such tall rectangles were described. 
For the purpose of simplifying
the computations, we start our proof in the half space 
model for $\ekt$,
i.e., $$\ekt = \big(\{(x,y,t)\in \R^3\mid y>0\},d\wt{s}_\tau^2\big),$$
where
$$d\wt{s}_\tau^2 = 
\frac{1}{y^2}(dx^2+dy^2)+\left(-\frac{2}{y}\tau dx+dt\right)^2.$$

In this model, Corollary~5.1 of~\cite{fp1}
implies that
for $h>\pi\sqrt{1+4\tau^2}$ and $m>0$, there exists
a minimal plane $S_h(m)$ with asymptotic boundary given by
the rectangle
at $\{y=0\}$ with the four vertices
$(-m,0,0)$, $(-m,0,h)$, $(m,0,0)$ and $(m,0,h)$,
see Figure~\ref{figrec}. Furthermore,
$S_h(m)$ is invariant under the
one-parameter subgroup of isometries of the half space model
which is generated by the hyperbolic isometries of $\hn2$ that
fix the points at infinity corresponding to
the vertical segments of $\pai S_h(m)$.

Note that the family of
hyperbolic translations $\{f_s(x,y,t) = (sx,sy,t)\}_{s>0}$
are isometries of $d\wt{s}_\tau^2$. Hence,
the image surfaces $\{f_s(S_h(m))\}_{s>0}$ give a foliation
of the open slab $\{(x,y,t)\mid y>0, 0<t<h\}$ by
minimal surfaces. This was proved
by Lima~\cite[Lemma~6]{Li} and follows from the
fact that for any $t_0\in (0,h)$, the intersection
$S_h(m)\cap \{t=t_0\}$ is a graphical arc of circle
with endpoints $(-m,0,t_0)$ and $(m,0,t_0)$.
Hence, it follows that each
$S_h(m)$ is area minimizing.

To prove the existence of $\cR_h(r)$ as claimed, we
just use an isometry between models as we next present.
Let $z = x+iy$ be a complex coordinate system
for $\R^2_+= \{(x,y)\in\R^2\mid y>0\}$ and let
$\phi\colon\R^2_+\to \bbD$ be the Möbius transformation
given by $\phi(z) = \frac{z-i}{z+i}$. Then,
\begin{equation}\label{eq:mudvar}
\psi(x,y,t) = \psi(z,t) = \left(\phi(z),t+4\tau\arctan\left(\frac{x}{y+1}\right)\right)
\end{equation}
is an isometry between the two models
of $\ekt$, $(\R^2_+\times \R,d\wt{s}_\tau^2)$ and $(\bbD\times\R,ds_\tau^2)$.
For a given $r\in(0,\pi)$, take
$m = \frac{\sin(r)}{1+\cos(r)}$ and let $\cR_h(r) = \psi(S_h(m))$.
It is straightforward to see that $\cR_h(r)$ has the asymptotic
boundary as claimed.
\end{proof}
\begin{figure}
\centering
\includegraphics[width=0.8\textwidth]{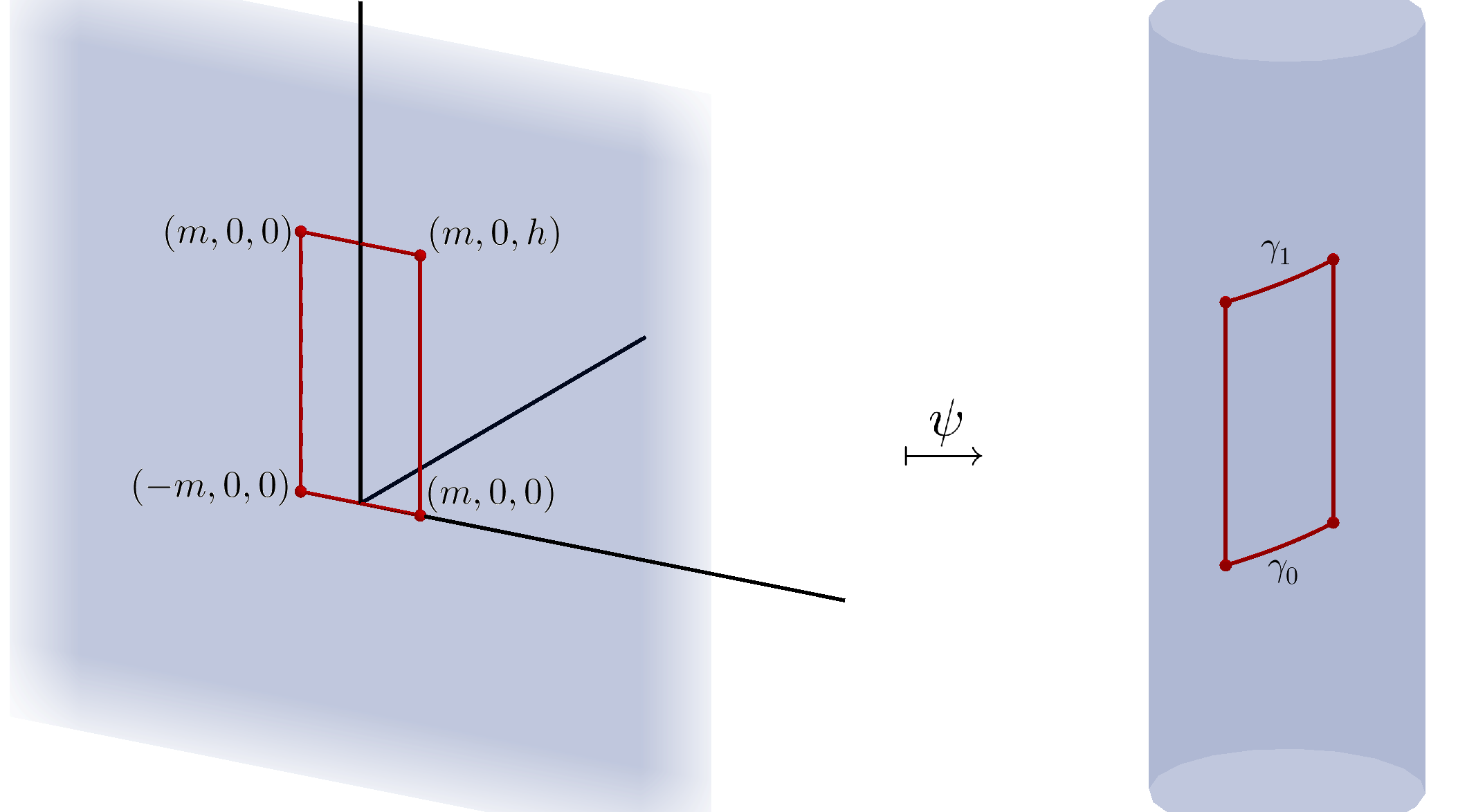}
\caption{$\pai S_h(m)$ is the rectangle in $\{y=0\}$
with vertices as above, and $\pai \cR_h(r)$
is the image of $\pai S_h(m)$ by $\psi$.\label{figrec}}
\end{figure}

The next result is a direct consequence of
Proposition~\ref{proptallrec}. We will make use of this result
in the proof of Theorem~\ref{thmmain}.
\begin{corollary}\label{corTallrec}
Given $t_1,t_2\in\R$ with $t_2>t_1+\pi\sqrt{1+4\tau^2}$,
there is $\de>0$ such that for any $\theta_1,\,\theta_2\in\sn1$
with $\abs{\theta_1-\theta_2}<\de$
there exists an area minimizing surface 
$\cR\subset \bbD\times(t_1,t_2)$ with
$\pai \cR\subset [\theta_1,\theta_2]\times(t_1,t_2)$.
\end{corollary}
\begin{proof}
This proof follows from the
fact that rotations about the $t$-axis and vertical translations
in $\bbD\times\R$ are isometries of $ds^2_\tau$, as we next explain.
Let $t_1$ and $t_2$ be as stated and let
$$h = \displaystyle\frac{t_2-t_1+\pi\sqrt{1+4\tau^2}}{2}, \quad\text{and}\quad
\ve = \displaystyle\frac{h - \pi\sqrt{1+4\tau^2}}{2}.$$
Let $\de>0$ be such that for any $\theta\in (-\de,\de)$
it holds that 
$4\tau\left|{\arctan\left(\frac{\sin(\theta)}{1+\cos(\theta)}\right)}\right|<\ve$.
Then, if $\theta_1,\,\theta_2 \in \sn1$ are such that
$\abs{\theta_1-\theta_2}<\de$, we may vertically translate and
rotate the surface $\cR_h(\abs{\theta_2-\theta_1})$ to find $\cR$
as claimed.
\end{proof}

\section{Existence and nonexistence results.}\label{secproofs}

We next prove the main results of the paper.
In Section~\ref{secexist} we prove that for any
tall curve $\G\subset \pai\ekt$
there exists an area minimizing surface $\S\subset \ekt$ with
$\pai \S = \G$. 
In Section~\ref{secnonex}, we prove that for certain short curves $\G$
with $h(\G)<(\sqrt{1+4\tau^2}-4\tau)\pi$ there is no area 
minimizing surface $\S$ with $\pai \S = \G$.

Throughout this section,
for any $t\in \R$, we let $P_t = \bbD\times \{t\}$
denote the horizontal plane at height $t$ in $\ekt$.

\subsection{The proof of Theorem~\ref{thmmain}.}\label{secexist}

First, we prove the theorem when $\G$ is a finite union
of disjoint parallel circles,
$$\G = \bigcup_{i=1,\cdots, n}\sn1\times\{h_i\},$$
where
$h_{i+1}-h_i >\pi\sqrt{1+4\tau^2}$ for all
$i\in\{1,\,\ldots,\,n-1\}$.
Note that each $P_{h_i}$ separates and is area minimizing. 
Hence, to show that $\cup_{i=1}^n P_{h_i}$ is area minimizing,
it suffices to prove it is the unique minimal surface in $\ekt$
with asymptotic boundary $\G$. We prove this by showing that 
there is no connected minimal surface in $\ekt$ with 
asymptotic boundary $\G$, when $n\geq 2$.

Suppose to the contrary that there is a connected minimal
surface $\S$ in $\ekt$ with $\pai \S = \G$.
For two consecutive $h_i<h_{i+1}$,
let $\de>0$ be such that
$(h_{i+1}-\de)-(h_i+\de)>\pi\sqrt{1+4\tau^2}$
and let $H_i = h_i+\de$ and $H_{i+1}=h_{i+1}-\de$.
Let $U = \bbD\times[H_i,H_{i+1}]$ be
the slab bounded by the planes $P_{H_i}$ and
$P_{H_{i+1}}$. Then
$U\cap \S$ is a compact minimal
surface that admits a connected component 
$\wh{\S}$ with $\partial \wh{\S} \subset P_{H_i}\cup P_{H_{i+1}}$,
$\partial \wh{\S} \cap P_{H_i}\neq \es$
and $\partial \wh{\S} \cap P_{H_{i+1}}\neq \es$.

Let $\{M_d\}_{d>0}$ be the family of rotational catenoids of
$\ekt$ given in Section~\ref{sec-cat},
vertically translated so that for all $d>0$,
$\pai M_d \subset \sn1\times (H_i,H_{i+1})$.
To obtain a contradiction, we now just recall that
when $d$ goes to infinity, the surfaces $M_d$
escape from any compact, and when $d$ approaches zero,
they converge (away from the origin, with multiplicity two) to
a horizontal plane. In particular, there must be a first contact 
point between $\wh{\S}$ and some $M_d$, which is 
a contradiction by the maximum principle.

Hence, we next proceed with the proof of Theorem~\ref{thmmain}
with the additional assumption that $\G$ is not a family of
parallel circles.

\begin{proof}[Proof of Theorem~\ref{thmmain}]
We start the proof by setting up the notation.
For each $n\in \N$ and $h\in \R$, let
$D_n(h)$ be the disk in the horizontal plane $P_h$
centered at the origin and with euclidean radius
$\tanh(n)$. In
particular, the family $\{D_n(h)\}_{n\in\N}$ gives an exhaustion
of $P_h$.
Also, for a given $T>0$, let
$\Delta_n(T) = \cup_{-T\leq h \leq T} D_n(h)$
be a compact solid cylinder in $\ekt$.
Since both horizontal planes and vertical planes over
complete geodesics are minimal
surfaces in the metric of $\ekt$, $\D_n(T)$ is mean convex
for all $n\in\N$ and $T>0$.

Let $\G$ be a tall curve in $\pai\ekt$ and
let $T>0$ be such that
$\G$ is contained in the open slab $\sn1\times(-T,T)$ of
$\pai\ekt$. For each $n\in\N$, let
$\G_n \subset \partial \D_n(T)$ be the radial projection
of $\G$ in $\partial \D_n(T)$.
Since $\D_n(T)$ is mean convex and $\G_n$ is 
an embedded, piecewise smooth curve in $\partial \D_n(T)$,
there exists an embedded, possibly disconnected,
area minimizing
surface $\S_n\subset \D_n(T)$ with $\partial \S_n = \G_n$.
Our next argument is to show that when $n\to \infty$,
then, up to a subsequence,
$\S_n$ converges to a nonempty complete surface 
$\S\subset \ekt$ such that $\pai \S = \G$.

Since each $\S_n$ is area minimizing, 
the number of connected components of
$\S_n$ is uniformly bounded by the number of connected components
of $\G_n$, which is equal to the number of components of $\G$.
In particular, we may pass to a subsequence to assume that 
there exists some $k\in\N$ such that the number
of connected components of each $\S_n$ is $k$, and
we let $\S_n^1,\ldots,\S_n^k$ denote such components,
labeled in such a way that for each $i\in\{1,2,\ldots,k\}$
the radial projection of $\partial \S_n^i$ to $\pai \ekt$
correspond to the same component of $\G$ for all $n\in \N$.
In particular, we just need to prove the result
when $k = 1$, since the general case follows from a finite diagonal
argument. Hence, from now on
we will assume that $\S_n$ is
connected, for all $n\in\N$.

Let $\O = \pai\ekt\setminus\G$. Since
$\G$ is tall, Corollary~\ref{corTallrec}
gives that for any $q\in \O$ there exists a tall rectangle
$\cR_q$ such that
$\pai\cR_q$ is disjoint from $\G$ and separates
$q$ from $\G$ in $\pai\ekt$. Let $U_q \subset \ekt$
be the region defined by $\cR_q$ in $\ekt$ such that
$q\in \pai U_q$.

We claim that $\S_n \cap U_q = \es$,
for all $n$ sufficiently large.
In the topology of
$\ol{\bbD}\times\R$, $\ol{U_q}$ and $\G$ are two disjoint
compact sets, and
the sequence $\G_n$ converges to 
$\G$. Hence there is $n(q)>0$ such that
for all $n\geq n(q)$,
$\G_n\cap \ol{U_q} = \es$, from where it follows that
$\G_n\cap U_q = \es$ in $\ekt$.

To prove the claim, we argue by contradiction and assume that
$\S_n$ intersects $U_q$ for some $n\geq n(q)$.
Now, a standard replacement argument yields a contradiction.
In fact, since $\G_n = \partial \S_n$ does 
not intersect $U_q$, then $S = \S_n\cap \ol{U_q}$ is a compact
smooth surface with boundary in $\partial U_q = \cR_q$.
Since $\cR_q$ is a topological plane,
there exists a compact subdomain $\wh{S}\subset \cR_q$
with $\partial \wh{S} = \partial S$.
Then, from the fact that both $\S_n$ and $\cR_q$ 
are area minimizing, we obtain that
${\rm Area}(S) = {\rm Area}(\wh{S})$. In 
particular, the compact surface defined by 
$$\S_n' = (\S_n\setminus S)\cup \wh{S}$$
is a nonsmooth area minimizing surface, a contradiction.

Next, we use the fact proved above to show that
the sequence $(\S_n)_{n\in\N}$
admits a limit point in $\ekt$; in other words, the
surfaces $\S_n$ do not {\em escape to infinity}.
Since $\G$ is not a finite collection of parallel circles,
there exists a horizontal plane $P_h$ in $\ekt$ such that
$\pai P_h$ intersects $\G$ transversely at some point $p$.
Hence, we may choose points $p_1,\,p_2 \in \pai P_h$ that
bound a closed arc $[p_1,p_2]\subset \pai P_h$
containing $p$ in its interior and such that
$[p_1,p_2]\cap \G =\{p\}$, see Figure~\ref{figPh}~(a).
Let $\g$ be a complete arc in $P_h$ with endpoints $p_1,p_2$
and let $A\subset P_h$ be the region bounded by $\g$ in
$P_h$ that contains $p$ in its asymptotic boundary.
Since $\G_n$ converges to $\G$, it follows that
$\G_n$ intersects $A$ transversely and only in one point,
for all $n$ sufficiently large.

The above argument shows that, when we consider
$[p_1,p_2]\cup \g$ as a simple closed curve in $\ol{\bbD}\times \R$,
the linking number between $[p_1,p_2]\cup \g$ and $\G_n$ is one,
for all $n$ sufficiently large. In particular, 
since $\partial \S_n = \G_n$, there must be
a point $q_n\in \g\cap \S_n$.
\begin{figure}
\centering
\begin{minipage}{0.48\textwidth}
\begin{center}
\includegraphics[width=0.85\textwidth]{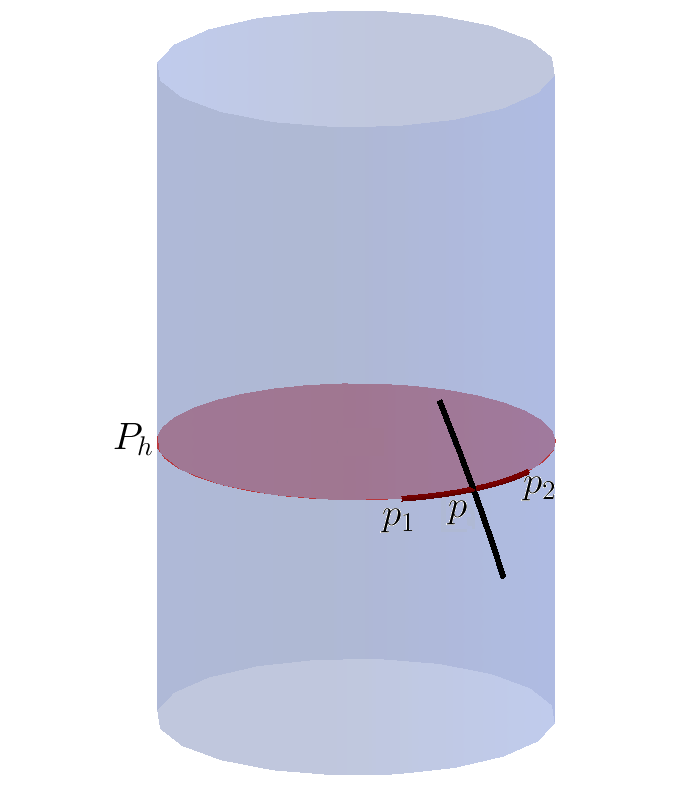}

(a)
\end{center}
\end{minipage}
\begin{minipage}{0.48\textwidth}
\begin{center}
\includegraphics[width=0.85\textwidth]{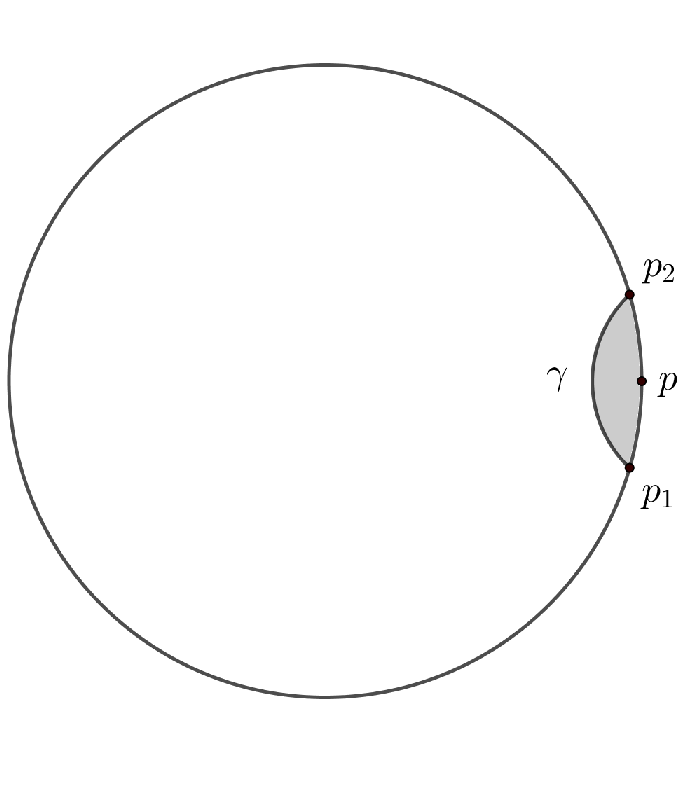}

(b)
\end{center}
\end{minipage}
\caption{
(a) shows the plane $P_h$ intersecting $\G$ transversely at $p$
and the arc $[p_1,p_2]\subset \pai P_h$.
In (b) we have the arc $\g\subset P_h$ and, highlighted, 
the region $A$.\label{figPh}}
\end{figure}

Let $U_1 = U_{p_1}$ and $U_2 = U_{p_2}$ be the respective 
regions bounded by two tall rectangles $\cR_{p_1}$
and $\cR_{p_2}$ as before. Then, for $n$ sufficiently
large, $\S_n\cap\left(U_1\cup U_2\right) = \es$.
Since $\g\setminus (U_1\cup U_2)$ is compact,
the sequence $\{q_n\}_{n\in\N}$ admits a convergent subsequence,
and then the surfaces $\S_n$ do not
escape to infinity. In particular, after passing
to a subsequence, it follows that $\S_n$ converges
(in the $C^{2,\alpha}$ topology on compacts of $\ekt$)
to a complete, area minimizing surface $\S\subset \ekt$.

It remains to prove that
$\pai \S = \G$. First, note that the fact that for any $p\in \O$
there exists $n(p)\in\N$ such that $\S_n \cap U_p = \es$
for all $n\geq n(p)$ gives immediately that $\pai \S \subset \G$.
Next, we show that given $p\in \G$, then $p\in \pai \S$.
First, assume that there is a plane $P_h\subset \ekt$ such that
$\pai P_h$ intersects $\G$ transversely at $p$. 
Take a sequence of
arcs $\g_n\subset P_h$ (each $\g_n$ resembles the arc $\g$ in
Figure~\ref{figPh} (b)) such that
the endpoints of $\g_n$ determine arcs in $\pai P_h$ that
intersect $\G$ uniquely at $p$ and such that
the respective regions $A_n\subset P_h$ bounded by $\g_n$
satisfy that $A_{n+1}\subset A_n$ and that 
$\cap_{n\in\N} \ol{A_n} = \{p\}$. The same arguments as above
give that for all $n\in \N$ there exists a point 
$q_n \in \S\cap \g_n$, from where it follows that 
$p = \lim_{n\to \infty} q_n \in \pai \S$.
Since the above argument is purely topological,
we notice that the general case when
the $t$-coordinate of $\G$ has a local extremal value 
at $p$ can be treated
in a similar manner, by considering a vertical plane instead of a 
horizontal one, and this finishes the proof of Theorem~\ref{thmmain}.
\end{proof}

\subsection{The proof of Theorem~\ref{thmnonexist}.}\label{secnonex}

In this section, we prove our nonexistence result
stated as Theorem~\ref{thmnonexist} in the Introduction.
The proof follows the ideas contained in Step~2 of the proof of
Theorem~2.13 in~\cite{cosk1}, with a few 
changes and necessary 
adaptations to the $\tau\neq 0$ setting.
A key step in the proof of our result is,
when $\tau \neq 0$, to show the existence of
a compact, connected area minimizing surface in $\ekt$
with boundary contained in two parallel planes that
are sufficiently far from each other. This
is stated in Proposition~\ref{propCatenoids} below 
and is proved in Section~\ref{secproof}, since
the arguments used in its proof are technical.

To what follows,
for each $t\in \R$ and $r>0$, we let
$$C_t(r) = \{(\tanh(r)\cos(u),\tanh(r)\sin(u),t)\mid u\in [0,2\pi)\}$$
be the circle in the horizontal plane $P_t$ (with coordinates 
given by the open disk $\bbD$)
centered at the origin with euclidean 
radius $\tanh(r)\in(0,1)$.

\begin{proposition}\label{propCatenoids}
For any $h\in(0,\frac{\pi}{2}\sqrt{1+4\tau^2})$
there exist $R>0$ and a compact, connected, 
area minimizing surface $S(h) \subset \ekt$
such that
$$\partial S(h) = C_h(R)\cup C_{-h}(R).$$
\end{proposition}

Assuming Proposition~\ref{propCatenoids},
we now proceed to the proof of Theorem~\ref{thmnonexist}.

\begin{proof}[Proof of Theorem~\ref{thmnonexist}]
Arguing by contradiction, let us assume that
$\G$ is a curve as stated and that
$\S$ is a complete, connected, area minimizing
surface in $\ekt$ such that
$\pai \S = \G$. Since $\S$ is area minimizing,
then $\S$ is properly embedded. In particular,
$\S$ is orientable and the fact that
$\S$ is connected implies that it separates 
$\ekt$ into two connected open regions $E_1,\,E_2$.

The asymptotic boundaries of $E_1$ and $E_2$ intersect along $\G$
and their union is the whole $\pai\ekt$. In particular,
if we let $\O_1 = {\rm int} (\pai E_1)$ and $\O_2 = {\rm int}(\pai E_2)$,
it follows that $\pai\ekt \setminus \G = \O_1\cup \O_2$.

After rotating $\S$ about the $t$-axis and
performing a vertical translation, 
the assumptions over $\G$ imply that there exist $\delta >0$ and
$T\in(0,\frac{\pi}{2}(\sqrt{1+4\tau^2}-4\tau))$ such that,
for all $\theta \in(-\de,\de)$ the
vertical segment $\g_{\theta} = \{e^{i\theta}\} \times[-T,T]$
intersects $\G$ transversely, in exactly two points, and both
points are interior to $\g_\theta$.
We may also reindex to assume, without loss of generality,
that $I_\de \times \{-T,T\}\subset \O_1$,
where $I_\de = \{e^{i\theta}\in \sn1 \mid \theta \in(-\de,\de)\}$;
see Figure~\ref{figshortc}~(a).

Let $V_1,\,V_2$ be open sets in $\ekt$
such that $\ol{V_1},\,\ol{V_2}\subset E_1$
and that 
$I_\de \times \{T\}\subset {\rm int }(\pai V_1)$,
$I_\de \times \{-T\}\subset {\rm int }(\pai V_2)$.
Also, let $V_3$ be another open set in $\ekt$
such that $\ol{V_3}\subset E_2$ and that
$\pai V_3 \subset \O_2$ with $(1,0)\in {\rm int}(\pai V_3)$.
For instance, we could take $V_1,\,V_2,\,V_3$ as sufficiently small
neighborhoods of $I_\de\times\{T\},\,I_\de\times\{-T\}$
and of $(1,0)$, respectively (see
Figure~\ref{figshortc}~(b)).

\begin{figure}
\noindent
\begin{minipage}{0.3\textwidth}
\begin{center}
\includegraphics[width=0.8\textwidth]{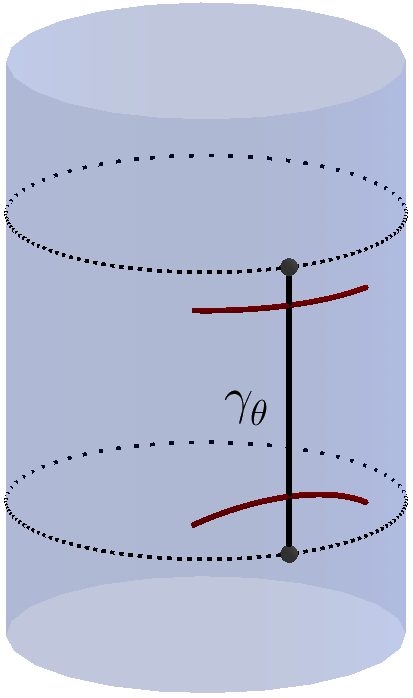}

(a)
\end{center}
\end{minipage}
\hfill
\begin{minipage}{0.3\textwidth}
\begin{center}
\includegraphics[width=0.8\textwidth]{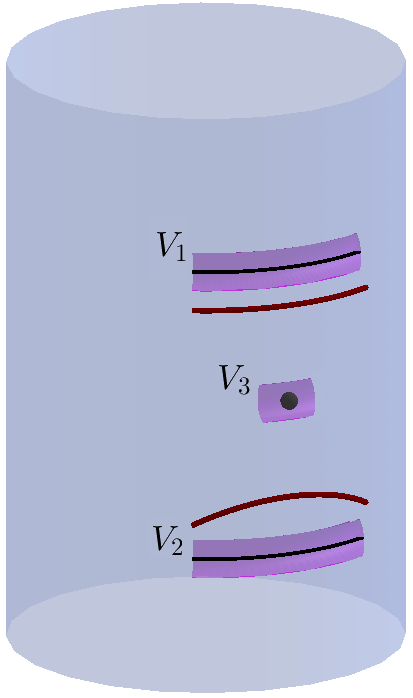}

(b)
\end{center}
\end{minipage}
\hfill
\begin{minipage}{0.3\textwidth}
\begin{center}
\includegraphics[width=0.8\textwidth]{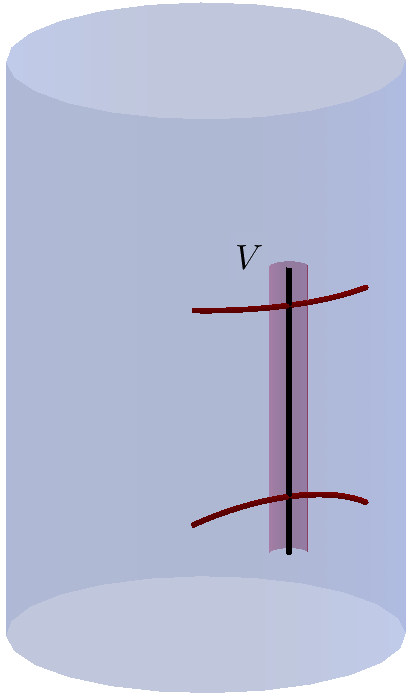}

(c)
\end{center}
\end{minipage}
\caption{In (a) we see one of the arcs $\g_\theta$ that has endpoints
in the circles $\sn1\times\{-T\},\,\sn1\times\{T\}$ and intersects
$\G$ transversely in two interior points.
In (b), we depict the neighborhoods $V_1,\,V_2$ and $V_3$ and
(c) shows the product neighborhood $V$ as a neighborhood of $\g_0$.
\label{figshortc}}
\end{figure}

Let, for $i=1,2,3$, $U_i = \pi_1(V_i)\subset \bbD$. Then
$U = U_1\cap U_2 \cap U_3$ is an open set of $\bbD$ that contains
$1$ in the interior of its asymptotic boundary.
Let $V = U\times (-T,T)\subset \ekt$ (see Figure~\ref{figshortc}~(c)).

From equation~\eqref{eq:boundheight}, we may
choose $h<\frac{\pi}{2}\sqrt{1+4\tau^2}$ such 
that $T<h-2\tau\pi$, and then 
Proposition~\ref{propCatenoids} implies the existence of
$R>0$ and of a connected, area minimizing surface $S = S(h)$
with boundary $C_h(R)\cup C_{-h}(R)$.
As before, let $D_t(R)$ denote the disk in $P_t$
centered at the
origin and with euclidean radius $\tanh(R)$.
By the maximum principle using horizontal and vertical
planes, we know that
$S\subset \cup_{t\in[-h,h]}D_t(R)$.
Furthermore, $\wh{S} = S\cup D_{-h}(R)\cup D_h(R)$
is a connected, embedded, compact surface in $\ekt$;
then $\wh{S}$ separates
$\ekt$ and defines a unique bounded region $A\subset \ekt$
with $\partial A = \wh{S}$.

For $r>0$, let $\varphi_r\colon \bbD\to\bbD$ be
the hyperbolic isometry of $\hn2$ that translates along
the geodesic given by the real axis and maps $0$ to $\tanh(r)$,
and let $\phi_r\colon \ekt\to \ekt$ be
its related isometry of $\ekt$ given by Corollary~\ref{cor-iso}.
Then,
for all $z\in\bbD$ and $t\in\R$,
$\pi_1(\phi_r(z,t)) = \varphi_r(z)$ and
\begin{equation}\label{eqarg}
\abs{\pi_2(\phi_r(z,t))-t} < 2\tau \pi.
\end{equation}

Let $S_r = \phi_r(S)$ and $A_r = \phi_r(A)$.
Then
$S_r$ is an area minimizing surface of $\ekt$ contained
in the boundary of the region $A_r$.
Let $D(R) = \pi_1(D_t(R))$ for $t\in \R$.
Since, when $r\to \infty$, $\varphi_r(D(R))$ is collapsing into 
the point at
infinity $z = 1$ and $U$ is an open set which contains 
$1$ in the interior of its asymptotic boundary, there exists 
$r_0>0$ such that for all $r\geq r_0$, $\pi_1(A_r)\subset U$.

Note that~\eqref{eqarg}, together with the
fact that $h-2\tau\pi >T$, gives that $\partial A_{r_0}$
intersects both regions $\bbD\times (-\infty,-T)$ and
$\bbD\times (T,+\infty)$. In particular, since
$\pi_1(A_{r_0})\subset U$ and $A_{r_0}$ is connected,
there exists a connected component $A_0$
of $A_{r_0}\cap V$ with boundary intersecting both 
$P_T\cap V_1$ and $P_{-T}\cap V_2$. In particular, 
$A_0$ intersects $E_1$.

On the other hand,
$V_3\cap V$ separates $V$ into two connected components that
intersect $A_0$, and then $V_3\cap A_0\neq \es$,
from where it follows that $A_0$ also intersects $E_2$
and that $A_0\setminus \S$ contains a compact
connected component with boundary contained
in $S_{r_0}\cup \S$.
Now, a standard replacement argument 
using that both $S_{r_0}$ and $\S$ are area minimizing
produces a nonsmooth area minimizing surface, which
is a contradiction that proves Theorem~\ref{thmnonexist}.
\end{proof}

\section{The proof of Proposition~\ref{propCatenoids}.}\label{secproof}

The proof of Proposition~\ref{propCatenoids},
which follows the ideas
presented in~\cite[Lemma~7.1]{cosk1},
will be carried out along
this section. For $d>0$, let $M_d$ be 
the rotational catenoid introduced in Section~\ref{sec-cat}.
The main idea here is
to show that for any $h\in(0,\frac{\pi}{2}\sqrt{1+4\tau^2})$
and $d$ sufficiently large, the surface
$$
M_d^h = M_d\cap (\bbD\times [-h, h])
$$
has less area than the union of the 
two disks in the parallel planes $P_{-h},P_h$
which share a boundary component with $M_d^h$.
Hence, the area minimizing surface 
with boundary $\partial M_d^h$ is necessarily connected.
Note that $M_d^h$ is compact if and only if 
$h <{\bf h}(d)$, which is equivalent to the existence of
$R>\arcsinh(d)$ such that $u_d(R) = h$.
For convenience,
for given $d>0$ and $R>\arcsinh(d)$, we define
$$M_d(R) = M_d^{u_d(R)} = M_d\cap\{-u_d(R)\leq t\leq u_d(R)\}.$$

Our first result is an area estimate for $M_d(R)$ for large
values of $d$. 
\begin{lemma}
\label{lemmaepsilon}
There exists $d_0>0$ such that
for any $d\geq d_0$ and $R>\arcsinh(d+1)$ it holds that
$${\rm Area}(M_d(R))
<2\pi\sqrt{1+4\tau^2}\left(\sqrt{e^{2R}-2-4d^2}+1\right).$$
\end{lemma}
\begin{proof}
Let $u\colon[a,b]\to\R$ be a smooth function and let
$\S$ be the rotational surface in $\ekt$ parameterized by
$$\S = \left\{(\tanh\left(r/2\right)\cos(\theta),
\tanh\left(r/2\right)\sin(\theta), u(r)) 
\mid r\in [a,b],\,\theta \in[0,2\pi)\right\}.$$
Then, a straightforward computation implies that the area of $\S$ is
given by
\begin{equation}\label{eq:areageneral}
{\rm Area}(\S)
= 2\pi\int_{a}^{b}
\sinh(s)\sqrt{(u'(s))^2+1+4\tau^2\tanh^2\left(s/2\right)}ds.
\end{equation}

Hence, it follows from~\eqref{eq:areageneral} and from~\eqref{defud}
that
$$
{\rm Area}(M_d(R)) =
4\pi\int_{\arcsinh(d)}^R
\sinh^2(s)\sqrt{\frac{1+4\tau^2\tanh^2\left(s/2\right)}
{\cosh^2(s)-1-d^2}}ds.$$
In particular, since $\abs{\tanh(x)}<1$ for all $x\in\R$, we obtain that
\begin{equation}\label{eqarea1}
{\rm Area}(M_d(R)) <
4\pi\sqrt{1+4\tau^2}\int_{\arcsinh(d)}^R
\frac{\sinh^2(s)}
{\sqrt{\cosh^2(s)-1-d^2}}ds.
\end{equation}

The next argument presents an adequate estimate
to the integral in~\eqref{eqarea1}, which we will denote by $I$.
Under the assumption that $R>\arcsinh(d+1)$, we may write
$I = I_1+I_2$ where
$$I_1
=
\int_{\arcsinh(d)}^{\arcsinh(d+1)}
\frac{\sinh^2(s)}
{\sqrt{\cosh^2(s)-1-d^2}}ds,\quad
I_2 =
\int_{\arcsinh(d+1)}^R
\frac{\sinh^2(s)}
{\sqrt{\cosh^2(s)-1-d^2}}ds.$$

To estimate $I_1$, we first use that $s<\arcsinh(d+1)$, obtaining
$$
I_1 \leq (d+1)
\int_{\arcsinh(d)}^{\arcsinh(d+1)}
\frac{\sinh(s)}
{\sqrt{\cosh^2(s)-1-d^2}}ds.
$$
Next, we use the change of variables $u = \cosh(s)$,
the identity
$\cosh(\arcsinh(x)) = \sqrt{1+x^2}$
and the fact
that for any $a \in \R$ the function $\log(\sqrt{x^2-a}+x)$
is a primitive to $\frac{1}{\sqrt{x^2-a}}$
to obtain that
\begin{eqnarray}
I_1
& \leq & 
(d+1)
\log\left(\frac{\sqrt{1+2d}+\sqrt{2+2d+d^2}}
{\sqrt{1+d^2}}\right)\nonumber\\
& < & 
(d+1)
\log\left(\frac{\sqrt{1+2d}+\sqrt{2+2d+d^2}}
{d}\right).\label{estimateI1}
\end{eqnarray}

In order to estimate $I_2$, we use the inequalities
$\sinh(x)<\frac{e^x}{2}$ and 
$\cosh^2(x)-1 > \frac{e^{2x}-2}{4}$, which hold for all $x\in\R$,
so that
$$
I_2 <
\int_{\arcsinh(d+1)}^R
\frac{e^{2s}}
{2\sqrt{e^{2s}-2-4d^2}}ds.
$$
Since 
$\frac{e^{2s}}
{2\sqrt{e^{2s}-2-4d^2}}$ is a primitive to
$\frac{\sqrt{e^{2s}-2-4d^2}}
{2}$ and
 $\arcsinh(x) = \log(\sqrt{x^2+1}+x)$, we obtain that
\begin{eqnarray}
I_2 & < &
\frac{\sqrt{e^{2R}-2-4d^2}-\sqrt{(\sqrt{2+2d+d^2}+d+1)^2-2-4d^2}}
{2}\nonumber\\
& < & 
\frac{\sqrt{e^{2R}-2-4d^2}-
\sqrt{
2(d+1)\sqrt{1+2d+d^2}+1+4d-2d^2}}
{2}\nonumber\\
& = & 
\frac{\sqrt{e^{2R}-2-4d^2}-
\sqrt{
8d+3}}
{2}.\label{estimateI2}
\end{eqnarray}

Using~\eqref{eqarea1},~\eqref{estimateI1} and~\eqref{estimateI2}
we obtain that, for any $d>0$ and $R>\arcsinh(d+1)$,
\begin{equation*}
\begin{array}{rcl}
\displaystyle\frac{{\rm Area}(M_d(R))}{4\pi\sqrt{1+4\tau^2}}
& < &
\displaystyle(d+1)
\log\left(\frac{\sqrt{1+2d}+\sqrt{2+2d+d^2}}
{d}\right)\\
&&+
\displaystyle\frac{\sqrt{e^{2R}-2-4d^2}-
\sqrt{
8d+3}}
{2}.
\end{array}
\end{equation*}

Since 
$$\lim_{d\to \infty}(d+1)
\log\left(\frac{\sqrt{1+2d}+\sqrt{2+2d+d^2}}
{d}\right)-
\frac{
\sqrt{
8d+3}}
{2} =0,$$
the lemma follows.
\end{proof}

Let $D_1(R)$ and $D_2(R)$ be the respective
disks in the horizontal planes $P_{u_d(R)}$ and $P_{-u_d(R)}$
such that $\partial (D_1(R) \cup D_2(R)) = \partial M_d(R)$.
Since vertical translations are isometries of $\ekt$, it
follows that
${\rm Area}(D_1(R)\cup D_2(R)) = 2{\rm Area}(D_1(R))$.
Using this fact, we prove the next result, which
compares the area of $M_d(R)$ with the area
of $D_1(R)\cup D_2(R)$ for $d$ and $R$ sufficiently large.

\begin{lemma}\label{lemma4}
Let $R(d) = \frac32\log(d)$. Then,
there exists $d_1>0$ such that
for all $d\geq d_1$ it holds that
\begin{equation}\label{eq:lower}
2{\rm Area}(D_1(R(d))) > 2\pi\sqrt{1+4\tau^2}(\sqrt{d^3-4}-\sqrt{d}).
\end{equation}
In particular, there exists $\wt{d}>0$ such that for all $d\geq \wt{d}$
it holds that 
\begin{equation}\label{eq:eqdolema}
{\rm Area}\big(D_1(R(d))\cup D_2(R(d))\big) > {\rm Area}(M_d(R)).
\end{equation}
\end{lemma}
\begin{proof}
For any $R>0$, we have that
\begin{equation*}
2{\rm Area}(D_1(R))
=4\pi\sqrt{1+4\tau^2}\int_{0}^{R}
\sinh(s)
\sqrt{\frac{\cosh(s)+\frac{1-4\tau^2}{1+4\tau^2}}{\cosh(s)+1}}ds.
\end{equation*}
In order to see this, just apply~\eqref{eq:areageneral}
for the function $u\equiv 0$, after using
the identity $\tanh^2\left(\frac{s}{2}\right) = \frac{\cosh(s)-1}{\cosh(s)+1}$.
Thus, if we denote $$\wt{I} =\displaystyle \int_{0}^{R}
\sinh(s)
\sqrt{\frac{\cosh(s)+\frac{1-4\tau^2}{1+4\tau^2}}{\cosh(s)+1}}ds,$$
it holds that
$2{\rm Area}(D_1(R)) =
4\pi\sqrt{1+4\tau^2}\ \wt{I}$.
Our next arguments estimate $\wt{I}$ from below.

First, use the substitution $u = \cosh(s)$ to
obtain that
$$\wt{I} =
\int_{1}^{\cosh(R)}
\sqrt{\frac{u+\frac{1-4\tau^2}{1+4\tau^2}}{u+1}}du.$$
Using that
for any $a\in\R$,
$$
\left.\left.
\frac{d}{d u}
\right(\sqrt{(u+a)(u+1)}+(a-1)\log\left(\sqrt{u+a}+\sqrt{u+1}\right)\right)
= \sqrt{\frac{u+a}{u+1}},
$$
we may compute $\wt{I}$ in terms of $R$ as follows
\begin{equation}\label{eq:itil}
\begin{array}{rcl}
\widetilde{I}& = &\sqrt{(\cosh(R)+\frac{1-4\tau^2}{1+4\tau^2})(\cosh(R)+1)}
-\displaystyle\frac{2}{\sqrt{1+4\tau^2}}\\
& \ & \\
&&-
\displaystyle\frac{8\tau^2}{1+4\tau^2}\log\left(\frac{
\sqrt{\cosh(R)+\frac{1-4\tau^2}{1+4\tau^2}}+\sqrt{\cosh(R)+1}}
{
\sqrt{\frac{2}{1+4\tau^2}}+\sqrt{2}}\right).
\end{array}
\end{equation}

Since $\frac{e^R}{2}<\cosh(R)<\frac{e^R+1}{2}$ and
$-1<\frac{1-4\tau^2}{1+4\tau^2}<1$, it follows 
from~\eqref{eq:itil} that
\begin{eqnarray*}
\wt{I} & > &
\sqrt{\left(\frac{e^R}{2}-1\right)
\left(\frac{e^R}{2}+1\right)}
-\frac{2}{\sqrt{1+4\tau^2}}\\
&&-
\frac{8\tau^2}{1+4\tau^2}\log\left(\frac{
\sqrt{\frac{e^R+1}{2}+1}+\sqrt{\frac{e^R+1}{2}+1}}
{
\sqrt{\frac{2}{1+4\tau^2}}+\sqrt{2}}\right)\\
& = &
\frac12\sqrt{e^{2R}-4}
-\frac{2}{\sqrt{1+4\tau^2}}
-
\frac{8\tau^2}{1+4\tau^2}\log\left(\frac{\sqrt{e^R+3}\sqrt{1+4\tau^2}}
{1+\sqrt{1+4\tau^2}}\right).
\end{eqnarray*}
Assuming that $R$ is large enough so $3<e^R$ and setting
$$c_1 = \frac{4\tau^2}{1+4\tau^2},\quad
c_2 = \frac{2}{\sqrt{1+4\tau^2}}
+\frac{8\tau^2}{1+4\tau^2}
\log\left(\frac{\sqrt{2}\sqrt{1+4\tau^2}}{1+\sqrt{1+4\tau^2}}\right),$$
we obtain that
$$\wt{I} > \frac12\sqrt{e^{2R}-4}-c_1R-c_2.$$
In particular, for $R(d) = \frac32\log(d)$,
it holds that
$$2{\rm Area}(D_1(R(d))) >
2\pi\sqrt{1+4\tau^2}\left(\sqrt{d^3-4}-3c_1\log(d)-2c_2\right),$$
for all $d$ sufficiently large.
To conclude the proof that~\eqref{eq:lower} holds, we just
notice that there exists $d_1>0$ large enough so that
for all $d\geq d_1$,
it holds that $3c_1\log(d)+2c_2 <\sqrt{d}$.

Now, to finish the proof of the lemma
it remains to show that there exists some
$\wt{d}$ such that~\eqref{eq:eqdolema} holds for all $d\geq \wt{d}$.
We first observe that
$\lim_{d\to \infty}\frac{R(d)}{\arcsinh(d+1)} =\frac32$, hence
there exists some $d_2>0$ such that for any $d\geq d_2$
it holds that $R(d)>\arcsinh(d+1)$.
Without loss of generality, we may assume that $d_2\geq d_0$,
where $d_0$ is defined in Lemma~\ref{lemmaepsilon}.
In particular, for all $d\geq d_2$
\begin{equation}\label{eq:upper}
{\rm Area}(M_d(R(d))) <2\pi\sqrt{1+4\tau^2}
\left(\sqrt{d^3-2-4d^2}+1\right).
\end{equation}

Now, we just use the fact that
$$\lim_{d\to\infty}\left(\sqrt{d^3-4}-\sqrt{d}-\sqrt{d^3-2-4d^2}-1\right) = \infty$$
to obtain some $\wt{d}\geq \max\{d_0,d_1,d_2\}$ such that for
all $d\geq \wt{d}$ 
$$\sqrt{d^3-2-4d^2}+1< \sqrt{d^3-4}-\sqrt{d}$$
and we can use~\eqref{eq:lower} and~\eqref{eq:upper} to finish the proof
of the lemma.
\end{proof}

At this point, we know that
for $d$ sufficiently large and $R(d) = \frac32\log(d)$, 
the area of $M_d(R)$ is less than the area of the two
horizontal disks with the same boundary as $M_d(R)$.
In particular, any area minimizing surface of $\ekt$
with boundary $\partial M_d(R)$ will be connected,
since the unique disconnected minimal surface with such boundary 
is $D_1(R)\cup D_2(R)$.

Our next result shows that for any
$h\in\left(0,\frac{\pi}{2}\sqrt{1+4\tau^2}\right)$ there exists
$d\geq \wt{d}$ so that
$M_d(R(d)) = M_d^h$, thereby completing the
proof of Proposition~\ref{propCatenoids}.

\begin{lemma}\label{finalprop}
For $R(d) = \frac32\log(d)$, it holds that
$$\lim_{d\to \infty} u_d(R(d)) = \frac{\pi}{2}\sqrt{1+4\tau^2}.$$
\end{lemma}
\begin{proof}
For a given $d>0$
let $\l_d\colon(\arcsinh(d),\infty)\to \R$ be defined by
\begin{equation}\label{defld}
\l_d(s) = \displaystyle{\int_{\arcsinh(d)}^s
\frac{d}{\sqrt{\sinh^2(r)-d^2}}dr.}
\end{equation}

It is immediate to obtain that for
all $R>\arcsinh(d)$ we have the inequality
$$u_d(R) >\sqrt{1+4\tau^2\tanh^2(\arcsinh(d)/2)}\ \lambda_d(R).$$
Since the results from Pe\~{n}afiel~\cite[Proposition~3.9]{p1}
imply that $u_d(R) < \frac{\pi}{2}\sqrt{1+4\tau^2}$
and $$\lim_{d\to \infty}\tanh^2(\arcsinh(d)/2) = 1,$$
in order to prove the lemma it suffices to show that
\begin{equation}\label{limi1}
\lim_{d\to \infty}\lambda_d\left(R(d)\right) 
= \frac{\pi}{2}.
\end{equation}
Equation~\eqref{limi1} is precisely
Lemma~7.3 of~\cite{cosk1}, but, for the sake of
completeness, we present its proof here.

We start with the change of variables
used in Proposition~5.2 of ~\cite{NSST},
$t =
\arccosh\left(\frac{\cosh(r)}{\sqrt{1+d^2}}\right)$,
so that~\eqref{defld} becomes
$$\lambda_d(s) = 
\int_0^{\arccosh\left(\frac{\cosh(s)}{\sqrt{1+d^2}}\right)}
\frac{d}{\sqrt{(1+d^2)\cosh^2(t)-1}}dt.$$
In particular, if we let
$\rho(s) = \arccosh(\sqrt{1+d^2}\cosh(s))$, it follows that
$$
\lambda_d(\rho(s)) = \int_0^s\frac{d}{\sqrt{(1+d^2)\cosh^2(t)-1}}dt.
$$

Let $s(d)$ be defined by $\rho(s(d)) 
= R(d)$. Then
\begin{equation}\label{eq:usel}
\lim_{d\to \infty}
\lambda_d(R(d)) = 
\lim_{d\to \infty}\int_0^{s(d)}\frac{d}{\sqrt{(1+d^2)\cosh^2(t)-1}}dt.
\end{equation}

For any $t>0$ it holds that
$$\frac{d}{\sqrt{1+d^2}}\frac{1}{\cosh(t)}\leq
\frac{d}{\sqrt{(1+d^2)\cosh^2(t)-1}}\leq \frac{1}{\cosh(t)},$$
and then
\begin{equation}\label{eq:desigint}
\frac{d}{\sqrt{1+d^2}}\int_0^{s(d)}\frac{1}{\cosh(t)}dt
\leq \l_d(R(d))
\leq \int_0^{s(d)}\frac{1}{\cosh(t)}dt.
\end{equation}

We note that $
s(d) = \arccosh\left(\frac{d^3+1}{2\sqrt{d^3}\sqrt{d^2+1}}\right)$;
in particular,
$\lim_{d\to\infty}s(d) = +\infty$. On the other hand,
$\int_0^{s(d)}\frac{1}{\cosh(t)}dt =
2\arctan\left(\tanh\left(\frac{s(d)}{2}\right)\right)$
from where it follows that
both lower an upper bounds on~\eqref{eq:desigint} converge to 
$\pi/2$, as $d\to \infty$.
This, together with~\eqref{eq:usel}, implies that
$$\lim_{d\to\infty}\l_d(R(d)) = \frac{\pi}{2},$$
which proves the lemma.
As already explained, this also finishes the 
proof of Proposition~\ref{propCatenoids}.
\end{proof}

\bibliographystyle{plain}
\bibliography{biblio}

\begin{thebibliography}{10}

\bibitem{CI}
J.~Castro-Infantes.
\newblock On the asymptotic plateau problem in $\widetilde{SL}_2(\mathbb{R})$.
\newblock Preprint avaiable at arXiv:2002.10911 [math.DG].

\bibitem{cosk1}
B.~Coskunuzer.
\newblock Minimal surfaces with arbitrary topology in
  $\mathbb{H}^2\times\mathbb{R}$.
\newblock Preprint avaiable at arXiv:1404.0214 [math.DG].

\bibitem{dan1}
B.~Daniel.
\newblock Isometric immersions into 3-dimensional homogeneous manifolds.
\newblock {\em Comment. Math. Helv.}, 82(1):87--131, 2007.

\bibitem{setou}
R.~Sa Earp and E.~Toubiana.
\newblock An asymptotic theorem for minimal surfaces and existence results for
  minimal graphs in $\mathbb{H}^2\times\mathbb{R}$.
\newblock {\em Math. Ann.}, (342):309--331, 2008.

\bibitem{FMMR}
L.~Ferrer, F.~Mart\'in, R.~Mazzeo, and M.~Rodr\'iguez.
\newblock Properly embedded minimal annuli in $\mathbb{H}^2\times\mathbb{R}$.
\newblock {\em Math. Ann.}, 375(1-2):541--594, 2019.

\bibitem{fp1}
A.~Folha and C.~Pe{\~{n}}afiel.
\newblock Minimal graphs in $\widetilde{PSL}(2,\mathbb{R})$.
\newblock {\em Mat. Contemp.}, (43):111--132, 2014.

\bibitem{KloMa}
B.~Kloeckner and R.~Mazzeo.
\newblock On the asymptotic behavior of minimal surfaces in {$\Bbb H^2\times
  \Bbb R$}.
\newblock {\em Indiana Univ. Math. J.}, 66(2):631--658, 2017.

\bibitem{Li}
V.~Lima.
\newblock The slab theorem for minimal surfaces in $\mathbb{E}(-1,\tau)$.
\newblock {\em Ann. Global Anal. Geom.}, 51:189--208, 2017.

\bibitem{NSST}
B.~Nelli, R.~Sa Earp, W.~Santos, and E.~Toubiana.
\newblock Uniqueness of ${H}$-surfaces in $\mathbb{H}^2\times\mathbb{R}$,
  $|{H}|\leq \frac12$.
\newblock {\em Ann. Global Anal. Geom.}, 33:307--321, 2008.

\bibitem{NeRo}
B.~Nelli and H.~Rosenberg.
\newblock Minimal surfaces in {${\Bbb H}^2\times\Bbb R$}.
\newblock {\em Bull. Braz. Math. Soc. (N.S.)}, 33(2):263--292, 2002.

\bibitem{p1}
C.~Pe{\~{n}}afiel.
\newblock Invariant surfaces in $\widetilde{PSL}_2(\mathbb{R},\kappa)$ and
  applications.
\newblock {\em Bull. Braz. Math. Soc. (N.S.)}, 4(43):545--578, 2012.

\bibitem{Younes}
R.~Younes.
\newblock Minimal surfaces in {$\widetilde{PSL_2(\Bbb R)}$}.
\newblock {\em Illinois J. Math.}, 54(2):671--712, 2010.

\end{thebibliography}

\begin{flushleft}
\textsc{Departamento de Matem\'atica}

\textsc{Universidade Federal de Santa Maria}

\textsc{Av. Roraima 1000, Santa Maria RS, 97105-900, Brazil}

\textit{Email:} patricia.klaser@ufsm.br
\end{flushleft}

\begin{flushleft}
\textsc{Mathematics Department}

 \textsc{Princeton University}

\textsc{Fine Hall, Washington Road, Princeton NJ, 08544, USA}

\textit{Email:} amenezes@math.princeton.edu
\end{flushleft}

\begin{flushleft}
\textsc{Departamento de Matem\'atica}

\textsc{Universidade Federal do Rio Grande do Sul}

\textsc{Av. Bento Gon\c{c}alves 9500, Porto Alegre RS, 91509-900, Brazil}

\textit{Email:} alvaro.ramos@ufrgs.br
\end{flushleft}

\end{document}